\documentclass[10pt]{amsart}
\usepackage{amsfonts,amssymb,amscd,amsmath,enumerate,verbatim,calc,amsthm}
\input xy
\xyoption{all}

\headsep=1cm \numberwithin{equation}{section}
\newtheorem{thm}{Theorem}[section]
\newtheorem{cor}[thm]{Corollary}

\newtheorem{prop}[thm]{Proposition}
\newtheorem{defn}[thm]{Definition}

\newcommand{\Ann}{\mbox{Ann}\,}
\newcommand{\coker}{\mbox{Coker}\,}
\newcommand{\Hom}{\mbox{Hom}\,}
\newcommand{\Ext}{\mbox{Ext}\,}
\newcommand{\Tor}{\mbox{Tor}\,}
\newcommand{\Spec}{\mbox{Spec}\,}

\newcommand{\Ker}{\mbox{Ker}\,}
\newcommand{\Ass}{\mbox{Ass}\,}

\newcommand{\Supp}{\mbox{Supp}\,}

\newcommand{\gr}{\mbox{grade}\,}
\newcommand{\depth}{\mbox{depth}\,}
\renewcommand{\dim}{\mbox{dim}\,}

\newcommand{\Tr}{\mbox{Tr}\,}

\newcommand{\Min}{\mbox{Min}\,}

\newcommand{\gd}{\mbox{G--dim}\,}
\newcommand{\SDE}{\mbox{SDE}\,}
\newcommand{\CME}{\mbox{CME}\,}

\newcommand{\h}{\mbox{ht}\,}

\newcommand{\E}{\mbox{E}}

\renewcommand{\H}{\mbox{H}}

\newcommand{\D}{\mbox{D}}

\newcommand{\fm}{\mathfrak{m}}
\newcommand{\fp}{\mathfrak{p}}

\newcommand{\fc}{\mathfrak{c}}

\begin{document}
\bibliographystyle{amsplain}


\title[Linkage of modules over Cohen-Macaulay rings]
 {Linkage of modules over Cohen-Macaulay rings}

\bibliographystyle{amsplain}

     \author[M. T. Dibaei]{Mohammad T. Dibaei$^{1}$}
     \author[M. Gheibi]{Mohsen Gheibi$^2$}
     \author[S. H. Hassanzadeh]{S. H. Hassanzadeh$^{3}$}
     \author[A. Sadeghi]{Arash Sadeghi$^4$}

\address{$^{1, 2, 3, 4}$ Faculty of Mathematical Sciences and Computer,
Tarbiat Moallem University, Tehran, Iran.}

\address{$^{1, 2, 4}$ School of Mathematics, Institute for Research in Fundamental Sciences
(IPM), P.O. Box: 19395-5746, Tehran, Iran.} \email{dibaeimt@ipm.ir}
\email{mohsen.gheibi@gmail.com} \email{sadeghiarash61@gmail.com}

\address{$^{3}$ Departamento de Mate\'{m}atica,CCEN,
Universidade Federal de Pernambuco,50740-540 Recife, PE
Brazil.} \email{hamid@dmat.ufpe.br}

\keywords{Linkage of modules,  Sliding Depth of Extention modules, modules with Cohen-Macaulay extension,
sequentially Cohen-Macaulay  \\
1. M.T. Dibaei was supported in part by a grant from IPM (No. 89130114)\\
 3. S.H.Hassanzadeh was partially supported by a grant from CNPq (Brazil).}



\begin{abstract} Inspired by the works in linkage theory of ideals, the concept of sliding depth of extension modules is defined to prove the Cohen-Macaulyness of linked module if the base ring is merely Cohen-Macaulay. Some relations between this new condition and other module-theory conditions such as G-dimension and sequentially Cohen-Macaulay are established. By the way several already known theorems in linkage theory are improved or recovered by new approaches.

\end{abstract}

\maketitle

\bibliographystyle{amsplain}
\section{introduction}
Classification is one of the main perspective in any field of
mathematics. Among rare theories deliberate this viewpoint in
commutative algebra and Algebraic Geometry, linkage theory is a
well-developed theory during decades. Classically it refers to
Halphen (1870) and M. Noether \cite{No}(1882) who worked to classify
space curves. During forties and fifties in twenty century Apery and
Ga\'{e}ta started the new contributions to classify curves in
$\mathbb{P}^3$, the reason one always feels that the twisted cubic
curve is that smooth is behind the fact that it is linked to a line.
In 1974 the significant work of Peskine and Szpiro\cite{PS} brought
breakthrough to this theory and stated it in the modern algebraic
language; two closed subschemes $V_1$ and $V_2$ in $\mathbb{P}^n$
are said to be linked if they are unmixed with no common component
and their union is complete intersection. More precisely two ideals
$I$ and $J$ in Cohen-Macaulay local ring $R$ is said to be linked if
there is a regular sequence $\alpha$ in their intersection such that
$I=\alpha:J$ and $J=\alpha:I$. The first main theorem in the theory
of linkage is the following \cite{PS}.

{\quote\it {Theorem A. If $(R,\fm)$ is a Gorenstein local ring, $I$
 and $J$ are two linked ideals of R then $R/I$ is Cohen-Macaulay if and only if $R/J$ is so.}}

Attempts to generalize this theorem lead to several development in
linkage theory, especially the works by C. Huneke and B. Ulrich \cite{Hu},\cite{HU1}. A
counterexample given by Peskine and Szpiro in the same article shows
that  {\it Theorem A} is no longer true if the base ring $R$ is only
Cohen-Macaulay (from now on CM). Trying to determine the accurate
condition for an ideal $I$, in a CM local ring, such that any ideal
which is linked to $I$ be CM, Huneke \cite{Hu} introduced the
concept of Strongly Cohen-Macaulay (SCM) condition, in the sense
that an ideal $I$ is SCM if all of the Koszul homology modules with
respect to some generating set of $I$ are Cohen-Macaulay. He stated
that {\it in a Cohen-Macaulay local ring any ideal linked to a
strongly Cohen-Macaulay  ideal is Cohen-Macaulay.} Herzog,
Vasconcelos and Villarreal \cite{HVV} replaced the  SCM condition by
the so called Sliding Depth condition; namely we say that an ideal
$I$ satisfies sliding depth condition if $\depth H_i \geq
\dim(R)-r+i$, for $i > 0$, where $H_i$ is $i$th homology of Koszul
with respect to some generating set of $I$ and $r$ is the number of
elements of this generating set.

The new progress in the linkage theory is the recent work  of  Martsinkovsky and Strooker
 \cite{MS} which established the concept of linkage of modules.
 This paper began to attract some interest, they not only could recover some
 of the known theorems in the theory of linked ideals such as the ones in \cite{Sc} but also present
  new conceptual ideas that only exist in module theory.

 In this paper, inspired by the works in the ideal case, we  extend the strongly Cohen-Macaulay and sliding depth condition
  for modules; so that we can state Theorem A for linked modules in
  CM local rings.

In section 2, as mentioned above, we define the new sliding depth
conditions for modules so called SDE (Sliding Depth on Ext´s) or CME
(Cohen-Macaulay Ext´s). Some sufficient condition for being SDE or
CME is given, for example in Proposition \ref{B0} it is shown that
Cohen-Macaulay $R$--modules with finite G--dimension are CME. As
well it is proven that over Cohen-Macaulay local ring $R$,  if $M$
is SDE, then $\lambda M$ is maximal Cohen-Macaulay(see Corollary
\ref{B1}).

Trying to detect the module theory invariants that have no ideal inscription in the theory of linkage of ideals,  we encounter to the combinatorial conception \emph{sequentially Cohen-Macaulay}. We first present a computational criterion for this concept involving the ideas from linkage of module in Corollary \ref{CC}. Finally in this section we pose an extension to a theorem of Foxby \cite{F} for the class of CME modules and answer this question; so that
in Cohen-Macaulay local ring
with canonical module $\omega_R$, it is shown that $M$ is CME if and
only if $M\otimes_R\omega_R$ is sequentially Cohen-Macaulay and
$\Tor_i^R(M,\omega_R)=0$ for $i>0$ ( Theorem \ref{A3}).

In section 3, for a finite $R$--module $M$ over Cohen-Macaulay local
ring $R$ of dimension $d\geq 2$ with canonical module $\omega_R$,
 we establish a duality between local cohomology modules of $M\otimes_R\omega_R$
  and those of $\lambda M$ (Theorem \ref{A4}) provided $M\otimes_R\omega_R$ be generalized
  Cohen-Macaulay. This theorem is a generalization to \cite[Theorem 10]{MS} and
  also \cite{Sc} while for its proof instead of techniques in derived category we appeal to Spectral sequences.
  Also whenever $M$ is generalized Cohen-Macaulay, under some vanishing assumption on Tor-modules of $M$ and $\omega_R$ we  show
that $\H^i_\fm(\lambda M)\cong \Ext^i_R(M,R)$ for $i=1,\ldots ,d-1$,
(Corollary \ref{B2}).

\newpage

\section{SDE and CME modules}
Throughout, $R$ is a Noetherian ring and $M$ is a finite generated
$R$--module.  Assume that $M$ is a stable $R$--module (i.e. $M$ has
no projective summand). Let
$P_1\overset{f}{\rightarrow}P_0\rightarrow M\rightarrow 0$ be a
finite projective presentation of $M$. The transpose $\Tr M$ of $M$
is defined to be $\coker f^*$ where $(-)^* := \Hom_R(-,R),$ which is
unique up to projective equivalence. Thus the minimal projective
presentations of $M$ represent isomorphic transposes of $M$ and it
is also stable $R$--module (see \cite[Theorem 32.13]{AF}). Let
$P\overset{\alpha}{\rightarrow}M$ be an epimorphism such that $P$ is
a projective. The syzygy module of $M$, denoted by $\Omega M$, is
the kernel of $\alpha$ which is unique up to projective equivalence.
Thus $\Omega M$ is determined uniquely up to isomorphism if
$P\rightarrow M$ is a projective cover. The operator $\lambda =
\Omega\Tr$, introduced by Martsinkovsky and Strooker, enabled them
to define linkage for modules: Two finitely generated $R$--modules $M$ and $N$ are said to
be\emph{ horizontally linked} if $M\cong \lambda N$ and
$N\cong\lambda M$. Thus, $M$ is horizontally linked (to $\lambda M$)
if and only if $M\cong\lambda^2M$. It is
shown
in \cite[Proposition 8 in section 4]{MS} that, over a Gorenstein local ring $R$, a stable $R$--module
$M$ with $\dim M= \dim R$ is maximal Cohen-Macaulay if and only if
$\lambda M$ is maximal Cohen-Macaulay and $M$ is unmixed.
 If the ring $R$ is merely a Cohen-Macaulay local ring this statements is not true \cite[section 6]{MS}.

The following definition is the module theory version of strongly Cohen-Macaulay and Sliding depth conditions.
\begin{defn}\label{D1}
 \emph{Let $R$ be a local ring of dimension $d$ and let $M$ be a finitely generated $R$--module.
 The module $M$ is called to be SDE (having Sliding Depth of Extension modules) if either $\Ext^i_R(M,R)=0$ or
 $\depth_R(\Ext^i_R(M,R))\geq d-i$ for all $i=1,\ldots ,d-1$. Also $M$ is called to be CME (having Cohen-Macaulay
 Extension modules) if either $\Ext^i_R(M,R)=0$ or $\Ext^i_R(M,R)$ is
Cohen-Macaulay of dimension $d-i$ for all $i=1,\ldots ,d-1$.}
\end{defn}
To see the ambiguity of CME-modules, it is shown in the next proposition that any Cohen-Macaulay module with finite
G-dimension is a CME. Clearly any CME module is  SDE.
 For the definition of G-dimension we refer to \cite{C}.

\begin{prop}\label{B0}
 Let $R$ be a Cohen-Macaulay local ring of dimension $d$. Then any  Cohen-Macaulay
$R$--module with finite G--dimension is
\emph{CME}.
\end{prop}
\begin{proof} Let $M$ be Cohen-Macaulay
$R$--module with finite G--dimension. The Auslander-Bridger formula
$\gd_R(M)+ \depth_R(M)=\depth R$ \cite[Theorem 1.4.8]{C} in
conjunction with the  Cohen-Macaulayness of $M$, imply that
$\gr_R(M)=\gd_R(M)=:g$. So that $\Ext^i_R(M, R)= 0$ for all $i\neq
g$. Choose $\underline{x}:=x_1, \ldots , x_g$ to be a maximal
$R$--sequence contained in $\Ann_R(M)$. We have
$\Ext^{g}_R(M,R)\cong \Hom_{R/(\underline{x})}(M,R/(\underline{x}))$
and $\Ext^i_{R/(\underline{x})}(M,R/(\underline{x}))=0$ for all $i>
0$. Since $M$ is a maximal Cohen-Macaulay
$R/(\underline{x})$--module, by \cite[Proposition 3.3.3]{BH},
$\Hom_{R/(\underline{x})}(M,R/(\underline{x}))$ is maximal
Cohen-Macaulay $R/(\underline{x})$--module. Therefore
$\Ext^{g}_R(M,R)$ is Cohen-Macaulay of dimension $d-g$.
\end{proof}

Determining the depth of linked ideals is in the center of the questions on the arithmetic properties of ideals. About linkage of modules the depth of modules linked to SDE modules is rather under control.

\begin{prop}\label{A1}
 Let $(R,\fm)$ be a local ring and let $M$ be a \emph{SDE} $R$--module. Then $\emph\depth_R(\lambda M)
\geq\min\{\emph\depth_R(M),\emph\depth R\}$.
\end{prop}
\begin{proof} Set $t=\min\{\depth_R(M),\depth_R(R)\}$. For $t=0$ it is trivial. Suppose that $t>0$.
Set $X:=\displaystyle\cup^{d-1}_{i=1}\Ass_R(\Ext^i_R(M,R))\cup
\Ass_R(M)\cup \Ass_R(R)$.  As $M$ is SDE and $t>0$, there is
$x\in\fm\setminus\underset{\fp\in X}{\cup}$$\fp$. Set
$\overline{M}=M/xM$ and $\overline{R}=R/xR$. The exact sequence
$0\rightarrow R\overset{x}{\longrightarrow}$$R\longrightarrow
\overline R\longrightarrow 0$
implies the exact sequence\\
\centerline{$0\longrightarrow M^* \overset{x}{\longrightarrow}$$ M^*
\longrightarrow\Hom_R(M,\overline{R}) \longrightarrow
\Ext^1_R(M,R)\overset{x}{\longrightarrow}$$\Ext^1_R(M,R)\longrightarrow
\cdots.$} As each map $\Ext^i_R(M,R)\overset{x}{\longrightarrow}$
$\Ext^i_R(M,R)$ is an injection for all $i=0, \cdots, d-1$, we have
standard isomorphisms $\Ext^i_{\overline{R}}(
\overline{M},\overline{R})\cong \Ext^i_R(M,\overline{R})\cong
\Ext^i_R(M,R)/x\Ext^{i}_R(M,R)$ for all $i=0,1,\ldots ,d-2$.
Therefore $\overline{M}$ is $\SDE$ as $\overline{R}$--module. Let
$P_1\longrightarrow P_0\longrightarrow M\longrightarrow 0$
 be a minimal projective presentation of $M$ and consider the exact sequence $0\longrightarrow M^*\longrightarrow
 P^*_0\longrightarrow \lambda M\longrightarrow 0$. As $\lambda M$ is a syzygy module, $x$ is also a non-zero-divisor on
 $\lambda M$. Thus there is a commutative diagram
with exact rows
$$\begin{CD}
&&&&&&&&\\
\ \ &&&& 0 @>>> M^*/{xM^*} @>>>P^*_0/{xP^*_0} @>>> {\lambda M}/{x\lambda M} @>>>0&  \\
&&&&&&  @VV{\cong}V @VV{\cong}V \\
\ \  &&&& 0 @>>>\Hom_{\overline R}(\overline M,\overline R) @>>> \Hom_{\overline R}(\overline P_0,\overline R)
 @>>>\lambda_{\overline{R}}\overline{M} @>>>0&\\
\end{CD}$$\\
which implies that $\lambda M/{x\lambda M} \cong \lambda_{\overline{R}}\overline{M}$.\\
By induction
$\depth_{\overline{R}}(\lambda_{\overline{R}}\overline{M})\geq
\min\{\depth_{ \overline{R}}(
\overline{M}),\depth_{\overline{R}}(\overline{R})\}=t-1$. Thus
$\depth_R(\lambda M)\geq t$.
\end{proof}

As a corollary of the above general proposition, we have the following generalization of\emph{ Theorem A} which is in fact the module theory version of \cite[Proposition1.1]{Hu}.

\begin{cor}\label{B1}
 Let $R$ be a Cohen-Macaulay local ring of dimension, and let $M$ be
a  maximal Cohen-Macaulay and $\emph{SDE}$ $R$--module.
 Then $\lambda M$ is maximal Cohen-Macaulay.
\end{cor}

The composed
functors $\mathcal{T}_i:=\Tr \Omega ^{i-1}$ for $i>0$ have been already
introduced by Auslander and Bridger \cite{AB} and recently used by
 Nishida \cite{N} to relate linkage and duality.  In the following
result, over Cohen-Macaulay
 local ring, we characterize an SDE module $M$ in terms of depths of the $R$-modules $\mathcal{T}_iM$ and
 $\lambda\Omega^iM$. Moreover, it follows that for $\lambda M$ to be
  maximal Cohen-Macaulay we only need $M$ to be SDE.

\begin{thm}\label{A2}
 Let $R$ be a Cohen-Macaulay local ring of dimension $d\geq2$, $M$ a finitely generated $R$--module.
 The following statements are equivalent.
\begin{itemize}
           \item[(i)]{$M$ is $\emph{SDE}$.}
            \item[(ii)]{$\emph{depth}_R(\mathcal{T}_iM)\geq d-i$ for all $i=1,\ldots ,d-1$.}
             \item[(iii)]{$\emph{depth}_R(\lambda\Omega^iM)\geq d-i$ for all $i=0,\ldots ,d-2$.}
          \end{itemize}
\end{thm}
\begin{proof}For a stable finite $R$--module $N$, there is an exact sequence (\cite[section 5]{MS}),
\centerline{$ 0\longrightarrow \Ext^1_R(\Tr N,R)\longrightarrow
N\longrightarrow \lambda ^2N\longrightarrow 0$.}
Note that since the
transpose of every finite $R$--module is either stable or zero,
$\Tr\mathcal{T}_iM$ is stably isomorphic to $\Omega^{i-1}M$ for
$i>0$, and $\Ext^1_R(\Omega^{i-1}M,R)\cong \Ext^i_R(M,R)$, so we
have the exact sequence
\begin{equation}\label{e} 0\longrightarrow \Ext^i_R(M,R)\longrightarrow \mathcal{T}_iM\longrightarrow \lambda^2
\mathcal{T}_iM\longrightarrow 0.
\end{equation} Also, since $\lambda ^2\mathcal{T}_iM$ is stably isomorphic to $\Omega\mathcal{T}_{i+1}M$ and $R$ is
 Cohen-Macaulay, we have \begin{equation}\label{e2} \depth_R(\lambda ^2\mathcal{T}_iM)=\depth_R
 (\Omega\mathcal{T}_{i+1}M).
 \end{equation}

(i)$\Longrightarrow$(ii). We proceed by induction on $i$. From the
exact sequence (\ref{e}) we have $\depth_R(\mathcal{T}_{d-1}M)\geq
1$. Now suppose that $i\leq d-2$ and
$\depth_R(\mathcal{T}_{i+1}M)\geq d-i-1$, accordingly,
$\depth_R(\Omega\mathcal{T}_{i+1}M)\geq d-i$ which in turn implies
$\depth_R(\mathcal{T}_iM)\geq d-i$, using  (\ref{e2}) and (\ref{e}).

(ii)$\Longrightarrow$(i). By (\ref{e2}) and the assumption,
$\depth_R(\lambda ^2\mathcal{T}_iM)=\depth_R
(\Omega\mathcal{T}_{i+1}M)\geq d-i$
 for all $i=1,\ldots ,d-1$. Using (\ref{e}), we get either $\Ext^i_R(M,R)=0$ or $\depth_R(\Ext^i_R(M,R))
 \geq d-i$ for all $i=1,\ldots ,d-1$.\\
(ii)$\Longleftrightarrow$(iii) Note that
$\Omega\mathcal{T}_{i+1}M=\lambda\Omega^iM$ for each $i$. Thus
$\depth_R(\mathcal{T}_iM)\geq d-i$
 for all $i=1,\ldots ,d-1$ if and only if $\depth_R(\lambda\Omega^iM)=\depth_R(\Omega\mathcal{T}_{i+1}M)
 \geq d-i$ for all $i=0,\ldots ,d-2$.
\end{proof}

A shellable simplicial complex is a special kind of Cohen-Macaulay complex
with a simple combinatorial definition. Shellability is a simple but
powerful tool for proving the Cohen-Macaulay property. A simplicial complex $\Delta$ is pure if each facet (= maximal
face) has the same dimension(cf. \cite[Section II]{S} ).

The concept of \emph{Sequentially Cohen-Macaulay} was defined by combinatorial commutative algebraists (\emph{loc. cit. 3.9} ) to answer
a basic question  to find a
"nonpure" generalization of the concept of a Cohen-Macaulay module,
so that the face ring of a shellable (nonpure) simplicial complex has this
property.

This concept was then applied by commutative algebraists to study some algebraic invariants or special algebras come from graphs(c.f. \cite{AH}).
In  following propositions we see the relation between \emph{Sequentially Cohen-Macaulay}, SDE and CME as well as a way to construct a family of modules with these properties.
\begin{defn}
 \emph{Let $(R,\fm)$ be a local Noetherian ring and let $M$ be a finitely generated
$R$--module. A finite filtration
$0=M_0\subset M_1\subset M_2\subset\ldots  \subset M_r=M $ of
submodules of $M$ is called a} Cohen-Macaulay filtration, \emph{if
each quotient $M_i/M_{i-1}$ is Cohen-Macaulay, and
$\dim_R(M_1/M_0)<\dim_R(M_2/M_1)<\ldots <\dim_R(M_r/M_{r-1})$. The
module $M$ is called} Sequentially Cohen-Macaulay \emph{if $M$
admits a Cohen-Macaulay filtration.}
\end{defn}
A basic fact about Sequentially Cohen-Macaulay modules is the following theorem of Herzog and Popescu
\cite[Theorem 2.4]{HP}.

\begin{thm}\label{SQ}
Let $R$ be Cohen-Macaulay local of dimension $d$ with canonical
module $\omega_R$. The following conditions are equivalent.
\begin{itemize}
           \item[(i)] $M$ is Sequentially Cohen-Macaulay.
            \item[(ii)] $\emph\Ext^{d-i}_R(M,\omega _R)$ are either 0
            or Cohen-Macaulay of dimension i for all $i\geq 0$.
          \end{itemize}
\end{thm}

 Thus one observes that over a Gorenstein local ring the conditions SDE, CME and Sequentially
 Cohen-Macaulay are equivalent. Hence Theorem \ref{A2} provides  the following computable characterization of
 sequentially Cohen-Macaulay modules.

\begin{cor}\label{CC}
 Let $R$ be Gorenstein local ring of dimension $d\geq 2$ and  $M$ be a  finitely generated
$R$--module. The following conditions are equivalent.
\begin{itemize}
           \item[(i)] $M$ is Sequentially Cohen-Macaulay.
            \item[(ii)] $\emph\depth_R(\lambda{\Omega^i M})\geq d-i
            $ for all $i$, $0\leq i\leq d-2$.
          \end{itemize}
\end{cor}

To prove Proposition \ref{P}, we need to recall the next generalization of the definition of linkage of modules, \cite[Definition 4]{MS}.

\begin{defn}
 \emph{Let $M$ and $N$ be two finitely generated
$R$--modules. The module $M$
 is said to be} linked \emph{to $N$ by an ideal $\fc$ of $R$, if $\fc \subseteq \Ann_R(M) \cap \Ann_R(N)$
 and $M$ and $N$ are horizontally linked as $R/\fc$--modules.}
\end{defn}
The following result shows that, over Gorenstein local ring, each of
the properties Sequentially Cohen-Macaulay, (or equivalently SDE or
CME) is preserved under evenly linkage.

\begin{prop}\label{P}
 Let $R$ be a Gorenstein local ring.
Then the condition sequentially Cohen-Macaulay (or equivalently,
\emph{SDE} or \emph{CME}) is preserved under evenly linkage by
ideals.
\end{prop}
\begin{proof} Set $d:=\dim R$. Let $\fc_1$ and $\fc_2$ be Gorenstein ideals.
Assume that $M_1$, $M$, and $M_2$ are $R$--modules such that $M_1$
is linked to $M$ by $\fc_1$ and $M$ is linked to $M_2$ by $\fc_2$.
For each $i
> 0$, by \cite[Lemma 11 and Proposition 16]{MS}, we have
\[\begin{array}{rl}
\Ext^{i+g}_R(M_1,R)&\cong \Ext^i_{R/\fc_1}(M_1,R/\fc_1)\\
&\cong\Ext^i_{R/\fc_2}(M_2,R/\fc_2)\\
&\cong \Ext^{i+g}_R(M_2,R)
\end{array}\]
 where $g=\h{\fc_1}=\h{\fc_2}=\gr_R{M_1}=\gr_R{M_2}$.

 Suppose that $M_1$ is Sequentially Cohen-Macaulay. By Theorem 2.9,
 $\Ext^{d-i}_R(M_1,R)$ is either zero or Cohen-Macaulay of dimension $i$ for each $i$. Hence
 $\Hom_{R/{\fc_1}}(M_1,{R/{\fc_1}})(\cong\Ext^g_R(M_1,R)) $ is Cohen-Macaulay $R$--module of dimension $d-g$,
 and so it is maximal Cohen-Macaulay $R/{\fc_1}$--module.
  Note that, for $j=1, 2$ there are exact sequences\\
  \centerline{$ 0 \rightarrow \Hom_{R/{\fc_j}}(M_j,{R/{\fc_j}}) \rightarrow P_j \rightarrow
  {\lambda_{R/{\fc_j}}M_j} \rightarrow 0 $} where $P_j$ is a
   projective ${R/{\fc_j}}$--module. Thus we have $\depth_{R/\fc_2}(\lambda_{R/{\fc_2}}M_2)= \depth_R(M)=
  \depth_R(\lambda_{R/{\fc_1}}M_1)\geq d-g-1$.
 Again $\Hom_{R/\fc_2}(M_2,{R/\fc_2})$
 is maximal Cohen-Macaulay $R/{\fc_2}$--module, and so that $\Ext^g_R(M_2,R)$
 is Cohen-Macaulay $R$--module of dimension $d-g$.
 Hence $M_2$ is Sequentially Cohen-Macaulay $R$--module by Theorem \ref{SQ}.
\end{proof}

As mentioned just after Theorem \ref{SQ}, over Gorenstein local
rings, CME modules are exactly sequentially Cohen-Macaulay modules.
On the other hand , when $(R,\fm)$ is Cohen-Macaulay  ring with
canonical module $\omega_R$, it follows from the result
\cite[Theorem 2.5]{F} of Foxby that $\Tor_i^R(M,\omega_R)=0$ for all
$i>0$, whenever $\gd_R M <\infty$.
 Moreover, Khatami and Yassemi in \cite[Theorem 1.11]{KY} prove that whenever $(R,\fm)$ is Cohen-Macaulay  ring with
canonical module $\omega_R$ and $M$ is an $R$-module with finite Gorenstein dimension then
 $M\otimes_R \omega_R$ is Cohen-Macaulay if and
only if $M$ is Cohen-Macaulay.  Note that by Lemma \ref{B0}, if
$\gd_R M <\infty$
 and $M$ is Cohen-Macaulay then $M$ is CME, i.e. the class of CME module contains
  the class of Cohen-Macaulay modules of finite G-dimensions.
 Hence the following question is naturally posed.

 {\it What does happen if in results of Foxby,  Yassemi and Khatami one replace
 finite G-dimension and Cohen-Macaulay conditions of $M$ with the condition
 that $M$ is \emph{CME}?}

 The following Theorem provides  an answer  to this question.

\begin{thm}\label{A3}
 Let $R$ be a Cohen-Macaulay local ring with the canonical
module $\omega_R$, and let $M$ be a finitely generated $R$--module.
Then the following two statements are equivalent.
\begin{itemize}
           \item[(i)] $M$ is \emph{CME}.
            \item[(ii)]{$M\otimes_R\omega_R$ is sequentially Cohen-Macaulay
            and $\emph{Tor}^R_i(M,\omega_R)=0$ for all $i>0$.}
          \end{itemize}
\end{thm}
\begin{proof} (i)$\Rightarrow$(ii). Let $P_\bullet: \cdots\rightarrow P_1\rightarrow
 P_0\rightarrow 0$ be a projective resolution of $M$,
and let $I^\bullet: 0\rightarrow I^0\rightarrow I^1\rightarrow
\cdots$ be an injective resolution of $\omega_R$ and construct the
third quadrant double complex
$F:=\Hom_R(\Hom_R(P_\bullet,R),I^\bullet).$ Let $^v\E$ (resp.
$^h\E$) denote the vertical (resp. horizontal) spectral sequence
associated to the double complex $F$. Then $^v\E^{i,j}_2
\cong\Ext^i_R(\Ext^j_R(M,R),\omega_R)$. Since $\Ext^i_R(M,R)$ is
either zero or is Cohen-Macaulay of dimension $d-i$, we have
$$^v\E^{i,j}_2\cong \left\lbrace
           \begin{array}{c l}
              \Ext^i_R(\Ext^j_R(M,R),\omega_R)\ \ & \text{ \ \ $i=j$,}\\
              0\ \   & \text{   \ \ $\textrm{otherwise}$.}
           \end{array}
        \right.$$\\
By using the equivalence of functors $\Hom_R(\Hom_R(X,R),Y)$ and
$X\otimes_RY$, when $X$(resp. $Y$) belongs to the subcategory of
projective (respectively injective)$R$--modules, we find that the
double complex $\Hom_R(\Hom_R(P_\bullet,R),I^\bullet)$ is isomorphic
to the third quadrant double complex $P_\bullet\otimes_RI^\bullet$.
Now we may use this double complex to find that
$$^h\E^{i,j}_2\cong \left\lbrace
           \begin{array}{c l}
              \Tor^R_i(M,\omega_R)\ \ & \text{ \ \ $j=0$,}\\
              0\ \   & \text{   \ \ $\textrm{otherwise}$.}
           \end{array}
        \right.$$\\
It follows that $^h\E_{\infty}=\ ^h\E_2$ and $^v\E_{\infty}=\
 ^v\E_2$. By comparing the two spectral sequences $^h\E$ and $^v\E$
we get $\Tor^R_i(M,\omega_R)= 0$ for all $i>0$. Thus there is a
filtration $0=\Phi_{d+1}\subset\Phi_d\subset\ldots
\subset\Phi_0=M\otimes_R\omega_R$ of $M\otimes_R\omega_R$ such that
 $\Ext^i_R(\Ext^i_R(M,R),\omega_R)\cong\Phi_i/\Phi_{i+1}$ for $i=0,\ldots ,d$. Note that, by
 \cite[Theorem 3.3.10]{BH}, $\Ext^i_R(\Ext^i_R(M,R),\omega_R)$
 is either zero or Cohen-Macaulay of dimension $d-i$. In other words $M\otimes_R\omega_R$ is
 sequentially  Cohen-Macaulay.\\

(ii)$\Rightarrow$(i). Consider the third quadrant double complex
$\Hom_R(P_\bullet \otimes_R \omega_R , E^\bullet )$. Using the same
notation as before, let $^vE$ (resp. $^hE$) be the vertical (resp.
horizontal) spectral sequences associated
 to the double complex
$\Hom_R(P_\bullet \otimes_R \omega_R , E^\bullet ).$  Then $^v\E^{i,j}_2\cong\Ext^i_R
(\Tor^R_j(M,\omega_R),\omega_R)\cong0$, for all $j>0$,
by our assumption. By using the equivalence of functors $\Hom_R(X\otimes_R\omega_R,Y)$
and $\Hom_R(X,\Hom_R(\omega_R,Y))$
in the category of $R$--modules,
we find the following isomorphism of double complexes $\Hom_R(P_\bullet\otimes_R\omega_R,E^\bullet)
\cong\Hom_R(P_\bullet,\Hom_R(\omega_R,E^\bullet)).$
Thus, we get $^hE^2_{i,j}\cong\Ext^i_R(M,\Ext^j_R(\omega_R,\omega_R))$, for all $i,j\geq0$.
As $\Ext^i_R(\omega_R,\omega_R)=0$ for $i>0$
and $\Hom_R(\omega_R,\omega_R)\cong R$, we get
$$^h\E^2_{i,j}\cong \left\lbrace
           \begin{array}{c l}
              \Ext^i_R(M,R)\ \ & \text{ \ \ $j=0$,}\\
              0\ \   & \text{   \ \ $\textrm{otherwise}$.}
           \end{array}
        \right.$$\\
As the two spectral sequences $^vE$ and $^hE$ collapse, we have
$^hE^\infty=\ ^hE^2$ and $^vE^\infty=\ ^vE^2$ and so that
$\Ext^i_R(M,R)\cong\Ext^i_R(M\otimes_R\omega_R,\omega_R)$. Since
$M\otimes_R\omega_R$ is sequentially Cohen-Macaulay,
$\Ext^i_R(M,R)=0$ or Cohen-Macaulay of dimension $d-i$
(see\cite[Theorem 1.9]{BH}), i.e. $M$ is $\CME$.
\end{proof}


  \begin{cor}
   Let $(R,\fm)$ be a Cohen-Macaulay local ring and let $M$ be a finitely generated
$R$--module. Set $\omega_{\widehat{R}}$ as the canonical module of
$\widehat{R}$, the completion of $R$ with respect to the $\fm$--adic
topology. Then the following are equivalent.
\begin{itemize}
           \item[(i)]{$M$ is $\emph{CME}$.}
            \item[(ii)]{$\widehat{M}\otimes_{\widehat{R}}\omega_{\widehat{R}}$ is sequentially Cohen-Macaulay and
            $\emph{Tor}^{\widehat{R}}_i(\widehat{M},\omega_{\widehat{R}})=0$ for all $i>0$.}
          \end{itemize}

\end{cor}


\section{local cohomology and linkage}
The main purpose of this section is to give a generalization of
\cite[Theorem 10]{MS} which states that $\H^i_\fm(\lambda M)\cong
\D(\H^{d-1}_\fm(M))$ for $i=1,\ldots ,d-1$, whenever $M$ is a
generalized Cohen-Macaulay module over Gorenstein local ring $R$,
where $\D(-)$ is the Matlis duality functor. Here we assume that $R$
is Cohen-Macaulay with canonical  module $\omega_R$ such that
$M\otimes_R\omega_R$ is generalized Cohen-Macaulay; it is then shown
that for each $i=1,\ldots ,d-1$,
$\H^i_{\fm}(M\otimes_R\omega_R)\cong \D(\H^{d-i}_{\fm}(\lambda M))$.
Also whenever $M$ is generalized Cohen-Macaulay, under some
vanishing assumption on Tor-modules of $M$ and $\omega_R$, we  show
that $\H^i_\fm(\lambda M)\cong \Ext^i_R(M,R)$ for $i=1,\ldots ,d-1$
(see Corollary \ref{B2}).

The next proposition will lead  to a "cohomologic criterion" for generalized Cohen-Macaulay modules to be linked  (Corollary \ref{Clink}). This proposition
has its own interest as it shows the exactness of the sequence \ref{mes}.
Although this exact sequence  may be already known, but for the sake of a detailed proof and statement we mention it.
\begin{prop}\label{A}
 Let $(R,\fm)$ be a Cohen-Macaulay local ring of dimension $d$ with
the canonical module $\omega_R$. Assume that $M$ is a finitely
generated $R$--module, $\emph\Ass_R(M)\subseteq \emph\Ass
R\cup\{\fm\}$ and that M satisfies the Serre condition $(S_2)$ on
the punctured spectrum. Set $M^\upsilon =\emph\Hom_R(M, \omega_R)$.
Let $\phi :M\longrightarrow M^{\upsilon\upsilon}$ be the natural
map, $K:=\emph\Ker(\phi)$ and $C:=\emph\coker(\phi)$. The following
statements holds true.
\begin{itemize}
\item[(i)] If $d=0$ then $K=0$.
\item[(ii)] If $d\leq 1$ then $C= 0$.
\item[(iii)] If $d\geq 1$ then $K\cong\Gamma_\fm(M)$.
\item[(iv)] If $d\geq 2$ then $C\cong\emph\H_\fm^1(M)$ and so there is an exact sequence
\begin{equation}\label{mes}
0\longrightarrow \Gamma_\fm(M)\longrightarrow M\longrightarrow
M^{\upsilon\upsilon}
       \longrightarrow \emph\H^1_\fm(M)\longrightarrow 0.
\end{equation}
\end{itemize}
\end{prop}
\begin{proof}  If $d= 0$, it is clear by \cite[Theorem
3.3.10]{BH} that $C=0 $ and $K= 0$. Assume that $d\geq 1$. One has
$\depth_R(M^{\upsilon\upsilon})\geq
\Min\{2,\depth_R(\omega_R)\}\geq 1 $ and so
$\Gamma_{\fm}(M^{\upsilon\upsilon})=0$. By applying
$\Gamma_{\fm}(-)$ on the exact sequence
 \begin{equation}\label{c} 0\longrightarrow K\longrightarrow
 M\overset{\phi}{\longrightarrow}
M^{\upsilon\upsilon}\longrightarrow C\longrightarrow
0,\end{equation} it follows that $\Gamma_{\fm}(K)=\Gamma_{\fm}(M)$.
Taking $d=1$, for each $\fp\in \Spec R\setminus \{\fm\}$,
$M_{\fp}\cong (M^{\upsilon\upsilon})_{\fp}\cong
(M_{\fp})^{\upsilon\upsilon}$, which implies that
$\Supp_R(K)\subseteq \{\fm\}$ i.e. $K=\Gamma_\fm(M)$. Hence we get
the exact sequence
\begin{equation}\label{a} 0\longrightarrow M/{\Gamma_{\fm}(M)}\longrightarrow
M^{\upsilon\upsilon}\longrightarrow C\longrightarrow0,\end{equation}
from which, by applying $\Gamma_{\fm}(-)$, we obtain the exact
sequence
\begin{equation}\label{b} 0\longrightarrow \Gamma_\fm(C)\longrightarrow
\H^1_{\fm}(M)\longrightarrow
\H^1_{\fm}(M^{\upsilon\upsilon})\longrightarrow
\H^1_{\fm}(C).\end{equation}
As $\depth_R(M^\upsilon)\geq\min\{2, \depth_R(\omega_R)\}\geq 1$,  $M^\upsilon$ is maximal Cohen-Macaulay.
Therefore the natural
map $M^\upsilon\longrightarrow
M^{\upsilon\upsilon\upsilon}$ is isomorphism. Using
the local duality theorem functorially gives the commutative diagram
$$\begin{CD}
&&&&\\
\ \  &&&&0 @> >>\Gamma_\fm(C)@>>>\H^1_{\fm}(M)@>>> \H^1_{\fm}(M^{\upsilon\upsilon})&  \\
 &&&&&&&&  @VV{\cong}V@VV{\cong}V \\
\ \  &&&&&&&& D(M^\upsilon)@>{\cong}>>D(M^{\upsilon\upsilon\upsilon}),&&\\
\end{CD}$$\\
where $D(-)=\Hom_R(-,E(R/\fm))$. Thus we get $\Gamma_\fm(C)=0$. Note
that if $\fp\in\Spec R\setminus\{\fm\}$, then $\dim R_\fp=0$
 and so $C_\fp=0=K_\fp$ by \cite[Theorem
3.3.10]{BH}. Hence $C=\Gamma_\fm(C)=0$.

 In case $d\geq 2$,
$\depth_R(M^{\upsilon\upsilon})\geq\min\{2,
\depth_R(\omega_R)\}\geq 2$ and (\ref{b}) implies that
\begin{equation}\label{d}\Gamma_\fm(C)\cong\H_\fm^1(M).\end{equation}
 Finally, we prove by induction
on $d\geq 2$ that $K=\Gamma_\fm(M)$ and $C=\H_\fm^1(M)$. Assume that
the statement is settled for rings with
 dimension smaller than $d$. Let $\fp\in\Supp_R(M)\setminus\{\fm\}$. We first show that
 $\fp\not\in\Supp_R(K)\cup\Supp_R(C)$.
  If $\h\fp=0$, the claim holds true as before. Assume that $\h\fp\geq1$. As $\dim R_\fp<d$,
  induction hypothesis for
$R_\fp$ implies that $K_{\fp}=\Gamma_{\fp R_{\fp}}(M_{\fp})$ and
$C_\fp=\H_{\fp R_\fp}^1(M_\fp)$. Since $\fp\not\in \Ass_R(R)$ and so
$\fp\not\in \Ass_R(M)$, we get $\depth_{R_\fp}(M_{\fp})\geq 1$ and
thus $K_{\fp}=0$, i.e. $\fp\not\in\Supp_R(K)$. For the case
$\h\fp=1$, we already have, $C_{\fp}=0$. Assume that $\h\fp\geq2$.
Again from the exact sequence \ref{c} and the fact that
$\depth_R(M^{\upsilon\upsilon})>1$, we get
$K=\Gamma_{\fm}(K)=\Gamma_{\fm}(M)$. As $R$ is Cohen-Macaulay and
$\Ass_R(M)\subseteq\Ass(R)\cup\{\fm\}$,
$\dim_{R_{\fp}}(M_{\fp})=\dim_{R_{\fp}}(R_{\fp})=\h{\fp}\geq 2$ and
so $\depth_{R_{\fp}}(M_{\fp})\geq
\Min\{2,\dim_{R_{\fp}}(M_{\fp})\}=2$ because $M$ satisfies $(S_2)$.
Hence $\H^1_{{\fp}R_{\fp}}(M_{\fp})=0$. Hence $C_{\fp}=0$, i.e.
$\fp\not\in\Supp_R(C)$. In particular, $K=\Gamma_\fm(K)$ and
$C=\Gamma_\fm(C)$. Now $K=\Gamma_\fm(M)$ by (\ref{c}), and
$C=\H_\fm^1(M)$ by (\ref{d}).
\end{proof}

\begin{cor}\label{Clink}
 Let $R$ be a Gorenstein local ring and let $M$ be a generalized Cohen-Macaulay stable $R$--module
with $\emph\dim_R(M)=\emph\dim R$. A necessary and sufficient
condition for $M$ to be horizontally linked is that
$\Gamma_\fm(M)=0$.
\end{cor}
\begin{proof} Note that, by \cite[Exescise 9.5.6]{BS}  , $\Ass_R(M)\subseteq \Ass_R(R)\cup \{\fm\}$ and $M$
satisfies $(S_2)$ on the punctured spectrum. As $M$ is linked  if
and only if the natural map $M\rightarrow M^{**}$ is one to one, the
result follows by Proposition \ref{A}.
\end{proof}

In the following result, we extend \cite[Theorm 10 in section
10]{MS} for Cohen-Macaulay rings with canonical module.

\begin{thm}\label{A4}
Let $(R,\fm)$ be local Cohen-Macaulay ring of dimension $d\geq 2$ with
canonical module $\omega_R$. Let $M$ be a finitely
generated $R$--module of dimension $d$, such that
$M\otimes_R\omega_R$ is generalized Cohen-Macaulay. Then for each
$i=1,\ldots ,d-1$, $\emph{H}^i_{\fm}(M\otimes_R\omega_R)\cong
\emph{Hom}_R(\emph{H}^{d-i}_{\fm}(\lambda M),E(R/{\fm}))$.
\end{thm}
\begin{proof} First we examine the general situation for an $R$--module $N$ which is a generalized
 Cohen-Macaulay of dimension
$d$. Let $0\rightarrow I^0\rightarrow I^1\rightarrow \cdots$ be an injective resolution
of $\omega_R$ and
 $\cdots\rightarrow P_1\rightarrow P_0\rightarrow 0$ be a projective resolution of $N$
  and construct the third quadrant double complex
  $F:=\Hom_R(\Hom_R(P_{\bullet},\omega_R),I^{\bullet})$. Let $^v\E$ (resp. $^h\E$)
  denote the vertical (resp. horizontal) spectral
  sequence associated to the double complex $F$. Then $^v\E^{i,j}_2
\cong\Ext^i_R(\Ext^j_R(N,\omega_R),\omega_R)$. As $N$ is generalized
Cohen-Macaulay, by local duality theorem, $\Ext^i_R(N,\omega_R)$ is
of finite length, for all $i=1,\ldots ,d$. Therefore
$$^v\E^{i,j}_2\cong \left\lbrace
           \begin{array}{c l}
              \Ext^i_R(N^\upsilon ,\omega_R) \ \   & \text{if  \ \ $j=0$,} \\
              \H^{d-j}_{\fm}(N) \ \ & \text{if  \ \ $j\neq 0 ,i= d$,}\\
              0\ \ & \text{if \ \ $j\neq 0\ ,i\neq d$.}
           \end{array}
        \right.$$\\
As the map $d^r$ is of bidegree $(r,1-r)$, one can observe that
$^v\E^{r,0}_r\cong \Ext^{d-r}_R(N^\upsilon,\omega_R)$, for
$r\geq 2$. Thus we have the following diagram:
$$
 \xymatrix{
 \Ext^0_R(N^\upsilon,\omega_R)\ar@{..>}^{d^d}[rrrrddd]& \cdots &\Ext^{d-2}_R(N^\upsilon,\omega_R)\ar^{d^2}[rrd] &\Ext^{d-1}_R(N^\upsilon,\omega_R) &\Ext^{d}_R(N^\upsilon,\omega_R) \\
 0                        &         &  &0         & \H^{d-1}_{\fm}(N)\\
 \vdots                   &         &  &          & \vdots\\
 0                        & \cdots  &  &0         & \H^{1}_{\fm}(N)\\
 0                        & \cdots  &  &0         & \H^{0}_{\fm}(N)\\
  }
 $$

To compute $^h\E_2$, we change our double complex with the functorial isomorphisms
$\Hom_R(\Hom_R(P_i,\omega_R),I^j)\cong P_i\otimes_R\Hom_R(\omega_R,I^j).$ Thus we get
      $$^h\E_2\cong\left\lbrace
           \begin{array}{c l}
              N \ \   & \text{if  \ \ $i=0, j=0$,} \\
              0\ \ & \text{otherwise}.
           \end{array}
        \right.$$\\
As $\Ker d^r$ and $\coker d^r$ are isomorphic to $^v\E_\infty$,
comparing the two spectral sequences, one get isomorphisms
$d^r:\Ext^{d-r}_R(N^\upsilon,\omega_R)\longrightarrow
\H^{d-r+1}_\fm(N)$ for $r=2,\ldots ,d-1$. Therefore
$\Ext^{d-r}_R(N^\upsilon,\omega_R)$ is of finite length and so,
by local duality theorem,
$\Ext^{d-r}_R(N^\upsilon,\omega_R)\cong
\mathrm{D}(\H^r_\fm(N^\upsilon)).$ Hence one obtains the
isomorphisms $\H^{d-r+1}_\fm(N)\cong
\mathrm{D}(\H^r_\fm(N^\upsilon))$, for all $r=2,\ldots ,d-1$.

Replacing $N$ by $M\otimes_R\omega_R$, gives
$$\H^{d-i+1}_\fm(M\otimes_R\omega_R)\cong \mathrm{D} \big( \H^i_\fm(
\Hom_R( M\otimes_R\omega_R,\omega_R) ) \big) \cong \mathrm{D}(
\H^i_\fm(M^*)),$$ for all $i=2,\ldots ,d-1.$ Consider the exact
sequence $0\rightarrow M^* \rightarrow P_0^* \rightarrow \lambda M
\rightarrow 0.$ Applying $\Gamma_\fm(-)$ we get
$\H^{i+1}_\fm(M^*)\cong \H^i_\fm(\lambda M)$ for $i=0,\ldots ,d-2$.
Therefore we have isomorphisms $\H^i_\fm(M\otimes_R\omega_R)\cong
\mathrm{D}(\H^{d-i}_\fm(\lambda M))$, for $i=2,\ldots ,d-1$. Now it
remains to prove the claim for $i=1$. Applying Theorem \ref{A} to
$M\otimes_R\omega_R$ and applying the functor $\Hom_R(-,\omega_R)$
on the exact sequence $0\rightarrow M^* \rightarrow P_0^*
\rightarrow \lambda M \rightarrow 0,$ we get the following
commutative diagram with exact rows and columns.
$$\begin{CD}
&&&&&&&&\\
\ \ &&&& P_0\otimes_R\omega_R @>>>M\otimes_R\omega_R @>>>0&  \\
&&&& @V{\cong}VV  @VVV \\
\ \  &&&& \Hom_R(P_0^*,\omega_R) @>>> {(M\otimes_R\omega_R)}^{\upsilon\upsilon}
 @>>> \Ext^1_R(\lambda M,\omega_R)@>>>0&\\
&&&&&&@VVV\\
\ \ &&&&&& \H^1_{\fm}(M\otimes_R\omega_R)&\\
&&&&&&@VVV \\
&&&&&&0 \\
\end{CD}$$\\
which implies that $\H^1_\fm(M\otimes_R\omega_R)\cong \Ext^1_R(\lambda M,\omega_R)\cong
\mathrm{D}(\H^{d-1}_\fm(\lambda M) ).$
\end{proof}
As the final result, we state the next corollary of Theorem \ref{A4}.
\begin{cor}\label{B2}
 Let $R$ be a Cohen-Macaulay ring of dimension $d\geq2$ with canonical
module $\omega_R$, and let $M$ be a finitely
 generated $R$--module. Suppose that $\emph{Ext}^i_R(M,R)$ is of finite length for $i=0,\ldots
 ,d-1$ and $\emph{Tor}^R_i(M,\omega_R)=0$
  for $i>0$. Then $\emph{H}^i_{\fm}(\lambda M)\cong \emph{Ext}^i_R(M,R)$, $i=1,\ldots ,d-1$,
  and so $\lambda M$ is generalized Cohen-Macaulay.
\end{cor}
\begin{proof} Since $\Tor^R_i(M,\omega_R)=0$ for all $i>0$, as in the proof of Theorem
\ref{A3} ((ii)$\Rightarrow$(i)),
$\Ext^i_R(M\otimes_R\omega_R,\omega_R)\cong \Ext^i_R(M,R)$ for
$i=0,\ldots ,d-1$. Hence $\Ext^i_R(M\otimes_R\omega_R,\omega_R)$ is
of finite length for $i=1,\ldots ,d-1$ and so that
$M\otimes_R\omega_R$ is generalized Cohen-Macaulay. Therefore the
result follows from local duality theorem and Theorem \ref{A4}.
\end{proof}

\bibliographystyle{amsplain}

\end{document}